\documentclass{article}
\usepackage{amssymb}
\usepackage{amsfonts}
\usepackage{amsmath}
\usepackage{graphicx}
\usepackage{mathrsfs}
\usepackage{graphicx}
\usepackage{epstopdf}
\usepackage[english]{babel}

\newtheorem{theorem}{Theorem}[section]

\newtheorem{corollary}[theorem]{Corollary}

\newtheorem{definition}[theorem]{Definition}

\newtheorem{lemma}[theorem]{Lemma}

\newtheorem{proposition}[theorem]{Proposition}
\newtheorem{remark}[theorem]{Remark}

\newenvironment{proof}[1][Proof]{\noindent\textbf{#1.} }{\ \rule{0.5em}{0.5em}}
\input{tcilatex}

\begin{document}

\title{Fundamental groups, homology equivalences and one-sided $h$-cobordisms%
}
\author{Yang Su, Shengkui Ye}
\maketitle

\begin{abstract}
We give a sufficient and necessary condition for the fundamental group
homomorphism of a map between CW complexes (manifolds) to induce partial
homology equivalences. As applications, we obtain characterizations of
fundamental groups of homology spheres and Moore manifolds. Moreover, a
classification of one-sided $h$-cobordism of manifolds up to diffeomorphisms
is obtained, based on Quillen's plus construction with Whitehead torsions.
\end{abstract}

\section{Introduction}

When studying manifold version of Quillen's plus construction, Guilbault and
Tinsley \cite{gt} introduce the notion of \textit{one-sided h-cobordism.}
This is important to their study of ends of non-compact manifolds (see
Guilbault and Tinsley \cite{gt2}). Recall that a one-sided $h$-cobordism $%
(W;X,Y)$ is a compact cobordism between closed manifolds such that the
inclusion $Y\hookrightarrow W$ is a homotopy equivalence. In \cite{Ye2}, the
second author introduces the notion of \textit{one-sided homology cobordism. 
} Let $(W;X,Y)$ be a compact cobordism between closed manifolds and $R$ be a 
$\mathbb{Z}[\pi _{1}(W)]$-module. We call $(W;X,Y)$ a \textit{one-sided }$R$%
\textit{-homology cobordism} if the inclusion $Y\hookrightarrow W$ induces
isomorphisms $\pi _{1}(Y)\cong \pi _{1}(W)$ and $H_{q}(Y;R)\cong H_{q}(W;R)$
for all $q\geq 0$. When $R=\mathbb{Z}[\pi _{1}(W)],$ the one-sided $R$%
-homology cobordism is a one-sided $h$-cobordism. There are two aims in this
article. The first is to give a sufficient and necessary condition for the
fundamental group homomorphism of a map between CW complexes (or manifolds)
to induce a one-sided $\mathbb{Z}$-homology cobordism. The second is to give
a classification of one-sided $h$-cobordism of manifolds up to
diffeomorphisms.

We study the case of CW complexes first. Let $f:X\rightarrow Y$ be a map
between CW complexes inducing $\mathbb{Z}$-homology equivalence. When $X$ is
fixed and $f$ induces an epimorphism of fundamental groups, Rodr\'{\i}guez
and Scevenels \cite{RS} show that the kernel $\ker :=\ker (\pi _{1}(f):\pi
_{1}(X)\rightarrow \pi _{1}(Y))$ is a relative perfect subgroup of $\pi
_{1}(X)$ \textsl{i.e.} $[\ker ,\pi _{1}(X)]=\ker .$ Moreover, the maximal
such kernel is the intersection of the transfinite lower central series of $%
\pi _{1}(X)$. When $\pi _{1}(f)$ is not necessarily an epimorphism,
Bousfield \cite[Lemma 6.1]{b} shows that there exists a CW complex $Y$ such
that $f$ is $\mathbb{Z}$-homology equivalent and $\pi _{1}(Y)=G$ if and only
if $H_{1}(f):H_{1}(X;\mathbb{Z})\rightarrow H_{1}(G;\mathbb{Z})$ is an
isomorphism and $H_{2}(f):H_{2}(X;\mathbb{Z})\rightarrow H_{2}(G;\mathbb{Z})$
is epimorphic. We consider high-dimensional homology equivalences, as
follows. Assume that for each integer $q\geq 2,$ $f_{q}:H_{q}(X;\mathbb{Z}%
)\rightarrow H_{q}(Y;\mathbb{Z})$ is an isomorphism (high-dimensional
homology equivalences). An immediate consequence is that $f$ induces an
epimorphism $H_{2}(f)\colon H_{2}(\pi _{1}(X);\mathbb{Z})\rightarrow
H_{2}(\pi _{1}(Y);\mathbb{Z})$ of second homology groups of fundamental
groups, which could be obtained from the Hopf exact sequence. If we fix the
CW complex $X$ and a group homomorphism $\alpha \colon \pi
_{1}(X)\rightarrow G,$ we show that the necessary condition just mentioned
is also sufficient for the existence of a CW complex $Y$ with $\pi _{1}(Y)=G$
and a cellular map $f:X\rightarrow Y$ inducing $\alpha $ and
high-dimensional homology equivalences. More precisely, we have the
following result.

\begin{theorem}
\label{main}Let $X$ be a (finite, \textrm{resp.}) CW complex and $R$ a
subring of rationals or the finite ring $\mathbb{Z/}p$ for some prime number 
$p.$ Suppose that $\alpha \colon \pi _{1}(X)\rightarrow G$ is a group
homomorphism from the fundamental group of $X$ to a (finitely presented%
\textrm{, resp.}) group $G$. Then the following are equivalent:

\begin{enumerate}
\item[(i)] $\alpha $ induces an epimorphism $H_{2}(\alpha )\colon H_{2}(\pi
_{1}(X);R)\rightarrow H_{2}(G;R).$

\item[(ii)] There exists a (finite, \textrm{resp.}) CW complex $Y$ and a
cellular map $f\colon X\rightarrow Y$ such that $\pi _{1}(Y)=G,$ $\pi
_{1}(f)=\alpha \colon \pi _{1}(X)\rightarrow \pi _{1}(Y)$ and for any
integer $q\geq 2,$ $f$ induces an isomorphism%
\begin{equation*}
f_{q}\colon H_{q}(X;R)\rightarrow H_{q}(Y;R).
\end{equation*}
\end{enumerate}
\end{theorem}

When $f$ is a homology equivalence, this clearly recovers the Bousfield's
result mentioned above (cf. \cite[Lemma 6.1]{b}).

In Ye \cite{Ye}, the second author shows that when $H_{1}(\alpha )\colon
H_{1}(\pi _{1}(X);\mathbb{Z})\rightarrow H_{1}(G;\mathbb{Z})$ is injective
and $H_{2}(\alpha )\colon H_{2}(\pi _{1}(X);\mathbb{Z})\rightarrow H_{2}(G;%
\mathbb{Z})$ is surjective, there is a CW complex $Y$ obtained by adding
low-dimensional cells to $X$ such that the fundamental group $\pi _{1}(Y)=G$
and the inclusion map $f:X\rightarrow Y$ induces the same fundamental group
homomorphism as $\alpha $ and for any integer $q\geq 2,$ the map $%
f_{q}:H_{q}(X;\mathbb{Z})\rightarrow H_{q}(Y;\mathbb{Z})$ is an isomorphism.
Actually in \cite{Ye}, more general coefficients are considered. Such a
construction gives a unified approach to Quillen plus construction,
Bousfield's integral homology localization, the existence of Moore spaces $%
M(G,1)$, Bousfield and Kan's partial $k$-completion of spaces and some
examples in the zero-in-the-spectrum conjecture. Compared with \cite[Theorem
1.1]{Ye}, in Theorem \ref{main} we drop the condition that $H_{1}(\alpha )$
is injective, but only for the cases when the coefficients $R$ are subrings
of the rationals or $\mathbb{Z}$/$p$ for some prime number $p.$ In these
cases, $R$ are principal ideal domains (PID). Therefore, all the
applications in \cite{Ye} when $R$ is a PID are corollaries of Theorem \ref%
{main} as well. These include Bousfield's integral homology localization and
the existence of Moore spaces $M(G,1)$. In \cite[Proposition 4.4]{Levin},
Levin proves, as emphasized by Dranishnikov \cite{Dra}, that for every
connected CW-complex $K$ there is a simply connected CW-complex $K^{+}$
obtained from $K$ by attaching cells of dimension $2$ and $3$ such that the
inclusion $K\rightarrow K^{+}$ induces isomorphisms of homology groups in
dimension$>1$. This is a special case of Theorem 1.1 when $R=\mathbb{Z}$ and 
$G=0$.

A further application is the following: let $n$ be a positive integer and $R$
be a subring of the rationals or the finite ring $\mathbb{Z/}p$ for some
prime number $p.$ We define an $R$-homology $n$-sphere to be a CW complex $Y$
with the same homology groups as those of the standard sphere $S^{n}$, i.e. $%
H_{\ast }(Y;R)=H_{\ast }(S^{n};R) $. When $R=\mathbb{Z}$, $n\geq 5$ and $%
Y=Y^{n}$ is a manifold, Kervaire \cite{kar} proves that a finitely presented
group $G$ is the fundamental group of a $\mathbb{Z}$-homology $n$-sphere $Y$
if and only if $H_{1}(G;\mathbb{Z})=H_{2}(G;\mathbb{Z})=0.$ The $\mathbb{Z}$%
-homology spheres are also studied by Dror \cite{dror}. The following result
gives a complete characterization of fundamental groups of $R$-homology
spheres for general coefficients $R$.

\begin{corollary}
\label{scor1} Suppose that $R$ is a subring of the rationals or the finite
ring $\mathbb{Z}/p$ for some prime number $p.$ Let $G$ be a (finitely
presented, \textrm{resp.}) group satisfying $H_{2}(G;R)=0$ and $X$ a
(finite, \textrm{resp.}) CW complex. There exists a (finite, \textrm{resp.})
CW complex $Y$ with $\pi _{1}(Y)=G$ and the homology group $H_{i}(Y;R)\cong
H_{i}(X;R),$ $i\geq 2,$ obtained from $X$ by attaching 1-cells, 2-cells and
3-cells. In particular, we have the following:

\begin{enumerate}
\item[(i)] A (finitely presented, \textrm{resp.}) group $G$ is the
fundamental group of an (finite, \textrm{resp.}) $R$-homology circle (i.e. $%
1 $-sphere) if and only if $H_{1}(G;R)=R$ and $H_{2}(G;R)=0$.

\item[(ii)] A group $G$ is the fundamental group of an $R$-homology $2$%
-sphere if and only if $H_{1}(G;R)=0$ and $H_{2}(G;R)$ is a quotient of $R$;

\item[(iii)] When $n>2$, a (finitely presented, \textrm{resp.}) group $G$ is
the fundamental group of an (finite, \textrm{resp.}) $R$-homology $n$-sphere
if and only if $H_{1}(G;R)=H_{2}(G;R)=0$.
\end{enumerate}
\end{corollary}

We now study the case of manifolds. The following result is a manifold
version of Theorem \ref{main}. (In the remainder of this paper, we assume
all manifolds are smooth manifolds, but our results hold in the PL and
topological categories as well.)

\begin{theorem}
\label{main2}Let $X$ be a closed manifold of dimension $n$ ($n\geq 5$), $G$
be a finitely presented group and $\alpha \colon \pi _{1}(X)\rightarrow G$ a
group homomorphism. Assume that $X$ is spin or that $\ker \{H_1(\alpha)
\colon H_{1}(\pi _{1}(X);\mathbb{Z})\rightarrow H_{1}(G;\mathbb{Z})\}$ is
torsion free. The following are equivalent:

\begin{enumerate}
\item[(i)] $\alpha $ induces an epimorphism $H_{2}(\alpha )\colon H_{2}(\pi
_{1}(X);\mathbb{Z})\rightarrow H_{2}(G;\mathbb{Z}).$

\item[(ii)] There exists a cobordism $(W;X,Y)$ such that $\pi _{1}(W)=G$ and
the inclusion map $g:X\hookrightarrow W$ induces the same fundamental group
homomorphism as $\alpha $, and for any integer $q\geq 2,$ the map $g$
induces an isomorphism 
\begin{equation*}
H_{q}(X;\mathbb{Z})\cong H_{q}(W;\mathbb{Z}).
\end{equation*}
\end{enumerate}
\end{theorem}

Compared with \cite[Theorem 1.1]{Ye2}, in Theorem \ref{main2} we drop the
spin condition on $X$ when $H_{1}(\alpha )$ is injective for the
coefficients $R=\mathbb{Z}$. Therefore, all the applications in \cite{Ye2}
when $R=\mathbb{Z}$ and $\ker (H_{1}(\alpha ))$ is a free abelian group are
corollaries of Theorem \ref{main2} as well. These include existence of
homology spheres, characterizations of high-dimensional knot groups and
integral homology localization of manifolds. As another application of
Theorem \ref{main2}, we give a characterization of fundamental groups of
Moore manifolds (see Section \ref{app} for more details).

The following corollary of Theorem \ref{main2} gives a characterization of
the fundamental groups of a one-sided $\mathbb{Z}$-homology cobordism.

\begin{corollary}
\label{mcor2}Let $X$ be a closed manifold of dimension $n$ ($n\geq 5$), $G$
a finitely presented group and $\alpha \colon \pi=\pi _{1}(X)\rightarrow G$
a group homomorphism. The following are equivalent.

\begin{enumerate}
\item[(i)] $H_{1}(\alpha ):H_{1}(\pi ;\mathbb{Z})\rightarrow H_{1}(G;\mathbb{%
Z})$ is an isomorphism and $H_{2}(\alpha ):H_{2}(\pi ;\mathbb{Z})\rightarrow
H_{2}(G;\mathbb{Z})$ is an epimorphism;

\item[(ii)] There exists a one-sided $\mathbb{Z}$-homology cobordism $%
(W;X,Y) $ with $\pi _{1}(W)=G$ and the inclusion $X\hookrightarrow W$
induces the same fundamental group homomorphism as $\alpha$.
\end{enumerate}
\end{corollary}

When $\alpha $ is an epimorphism, Corollary \ref{mcor2} is the integral
localization of manifolds (cf. Corollary \ref{integral}), which was first
proved by the second author in \cite{Ye2}.

While it seems complicated to give a classification of one sided $R$%
-homology cobordisms for a general module $R,$ we give a classification of
one-sided $h$-cobordisms up to diffeomorphisms. Two one-sided $h$-cobordisms 
$(W;M,N)$ and $(W^{\prime };M,N^{\prime })$ are equivalent if there exists a
diffeomorphism $f\colon W\rightarrow W^{\prime }$ such that $f|_{M}=\mathrm{%
id}_{M}$ and $f(N)=N^{\prime }$. Clearly this is an equivalence relation.
Denote by $S_{h}(M)$ the set of all equivalence classes of one-sided $h$%
-cobordism $(W;M,N)$ on $M$. We have the following result.

\begin{theorem}
\label{th3}Let $M^{n}$ be a manifold of dimension $n\geq 5.$ Denote by $%
\mathrm{Pf(}\pi _{1}(M)\mathrm{)}$ the set of all normally finitely
generated perfect normal subgroups in $\pi _{1}(M).$ Then there is a
bijection of sets%
\begin{equation*}
S_{h}(M)\cong \bigcup_{P\in \mathrm{Pf(}\pi _{1}(M)\mathrm{)}}\mathrm{Wh}%
(\pi _{1}(M)/P),
\end{equation*}%
where $\mathrm{Wh}(-)$ is the Whitehead group.
\end{theorem}

\medskip

The proof of Theorem \ref{th3} is based on a manifold version of Quillen's
plus construction with a given Whitehead torsion (see Section \ref{qm} for
more details).

The paper is organized as follows. In Section 2, we prove Theorems \ref{main}%
, \ref{main2} and list some applications. In Section 3, we introduce
Quillen's plus construction with Whitehead torsions for CW complexes and
manifolds. Theorem \ref{th3} is proved at the end of Section 3.

\medskip

\noindent \emph{Acknowledgements} S.-K.Y.~would like to thank the Chinese
Academy of Sciences for a visit during fall 2011 by which the research of
this paper was initiated. Y.S.~was partially supported by NSFC (grand No.
10901151); S.-K.~Y. was partially supported by RDF-13-02-08 of Xi'an
Jiaotong-Liverpool University. The authors would like to thank the referees
for helpful suggestions for improving the paper. The authors would also like
to thank C.~R.~Guilbault for his comments on an earlier version of this
paper, and thank Dan Chen for drawing the picture.

\section{Homology equivalences and fundamental groups}

In this section, we prove Theorem \ref{main} and \ref{main2}. Some
applications are also given. These include the integral homology
localizations, existence of Moore spaces, homology spheres and
high-dimensional knot groups.

\subsection{Proof of Theorems \protect\ref{main} and \protect\ref{main2}}

\begin{proof}[Proof of Theorem \protect\ref{main}]
Suppose that $f:X\rightarrow Y$ is a map such that for any integer $q\geq 2,$
the map $f_{q}:H_{q}(X;R)\rightarrow H_{q}(Y;R)$ is an isomorphism.
According to Hopf's exact sequence (cf.~\cite[Lemma 2.2]{hig}), we have the
following diagram%
\begin{equation*}
\begin{array}{ccc}
H_{2}(X;R) & \twoheadrightarrow & H_{2}(\pi _{1}(X);R) \\ 
\downarrow &  & \downarrow \\ 
H_{2}(Y;R) & \twoheadrightarrow & H_{2}(\pi _{1}(Y);R).%
\end{array}%
\end{equation*}%
Therefore, the group homomorphism $\alpha $ induces a surjection $H_{2}(\pi
_{1}(X);R)\rightarrow H_{2}(G;R)=H_{2}(\pi _{1}(Y);R).$

Conversely, suppose that $\alpha $ induces an epimorphism $H_{2}(\pi
_{1}(X);R)\rightarrow H_{2}(G;R).$ The strategy of constructing $Y$ is
similar to that of \cite[Theorem 1.1]{Ye}. For the group homomorphism $%
\alpha :\pi _{1}(X)\rightarrow G,$ we can construct a CW complex $W$ by
adding $1$-cells and $2$-cells to $X$ such that $\pi _{1}(W)=G$, just as in
the proof of \cite[Theorem 1.1]{Ye}. We consider the homology groups of the
pair $(W,X).$ By Hopf's exact sequence, there is a commutative diagram%
\begin{equation*}
\begin{array}{cccc}
H_{2}(\widetilde{X})\otimes _{\mathbb{Z}[G]}R & \rightarrow & H_{2}(%
\widetilde{W})\otimes _{\mathbb{Z}[G]}R &  \\ 
\downarrow &  & \downarrow j_{4} &  \\ 
H_{2}(X;R) & \overset{j_{2}}{\longrightarrow } & H_{2}(W;R) & \overset{j_{1}}%
{\rightarrow }H_{2}(W,X;R)\rightarrow H_{1}(X;R)\rightarrow H_{1}(W;R). \\ 
\downarrow j_{3} &  & \downarrow j_{5} &  \\ 
H_{2}(\pi ;R) & \overset{\alpha _{\ast }}{\longrightarrow } & H_{2}(G;R) & 
\end{array}%
\end{equation*}%
Since $R$ is a principal ideal domain, the relative homology group $%
H_{2}(W,X;R)$ is a free $R$-module and the image $\mathrm{im}j_{1}$ is also
a free $R$-module. By diagram chasing (cf. \cite[Theorem 1.1]{Ye}, proof of
Theorem 1.1), there is a surjection 
\begin{equation*}
j_{1}\circ j_{4}\colon H_{2}(\widetilde{W})\otimes _{\mathbb{Z}%
[G]}R\rightarrow \mathrm{im}j_{1}.
\end{equation*}%
Note that $R$ is a $G$-dense ring in the sense of \cite{Ye}. Therefore, we
can find a basis $S$ for $\mathrm{im}j_{1}$ in the image of $H_{2}(%
\widetilde{W})\otimes 1.$ Then there are maps $b_{\lambda }:S_{\lambda
}^{2}\rightarrow W$ with $\lambda \in S$ such that the composition of maps 
\begin{equation*}
H_{2}(\vee _{\lambda \in S}S_{\lambda }^{2};R)\rightarrow
H_{2}(W;R)\rightarrow \mathrm{im}j_{1}
\end{equation*}%
is an isomorphism$.$

For each such $\lambda $, attach a 3-cell $(D^{3},S^{2})$ to $W$ along $%
b_{\lambda }$. Let $Y$ denote the resulting space. We see that the diagram 
\begin{equation*}
\begin{array}{ccc}
\vee _{\lambda }S^{2} & \longrightarrow & W \\ 
\downarrow & \ulcorner _{\cdot } & \downarrow \\ 
\vee _{\lambda }D^{3} & \longrightarrow & Y%
\end{array}%
\end{equation*}%
is a pushout diagram. By the van Kampen theorem, the fundamental group of $Y$
is still $G.$ We have a commutative diagram%
\begin{equation*}
\begin{array}{ccccc}
H_{2}(X;R) & \rightarrow & H_{2}(W;R) & \overset{j_{1}}{\rightarrow } & 
H_{2}(W,X;R) \\ 
\downarrow &  & \downarrow &  & \downarrow \\ 
H_{2}(X;R) & \rightarrow & H_{2}(Y;R) & \overset{b}{\rightarrow } & 
H_{2}(Y,X;R)%
\end{array}%
\end{equation*}%
Since the relative homology group $H_{2}(Y,W;R)=H_{2}(\vee _{\lambda
}D^{3},\vee _{\lambda }S^{2};R)=0,$ the map $H_{2}(W;R)\rightarrow
H_{2}(Y;R) $ is surjective. Therefore, the right vertical map induces a
surjection $\mathrm{im}j_{1}\rightarrow \mathrm{im}b$. Denoting by $H_{\ast
}(-)$ the homology groups $H_{\ast }(-;R),$ we have the following
commutative diagram:%
\begin{equation*}
\begin{array}{ccccccc}
\cdots \rightarrow H_{3}(\vee D^{3},\vee S^{2}) & \rightarrow & H_{2}(\vee
S^{2},\mathrm{pt}) & \rightarrow & H_{2}(\vee D^{3},\mathrm{pt}) & 
\rightarrow & H_{2}(\vee D^{3},\vee S^{2}) \\ 
\downarrow &  & \downarrow &  & \downarrow &  & \downarrow \\ 
\cdots \rightarrow H_{3}(Y,W) & \rightarrow & H_{2}(W,X) & \rightarrow & 
H_{2}(Y,X) & \rightarrow & H_{2}(Y,W).%
\end{array}%
\end{equation*}%
Since $H_{2}(\vee S^{2},\mathrm{pt})\rightarrow \mathrm{im}j_{1}$ is an
isomorphism and $H_{2}(\vee D^{3},\mathrm{pt})=0,$ the image $\mathrm{im}%
b=0. $ By a five lemma argument, for any $i\geq 3$ the relative homology
group $H_{i}(Y,X;R)=0.$ The vanishing of these relative homology groups and $%
\mathrm{im}b$ shows that for any $q\geq 2,$ there is an isomorphism $%
H_{q}(X;R)\cong H_{q}(Y;R).$
\end{proof}

\bigskip

The proof of Theorem \ref{main2} is parallel to that of Theorem \ref{main}
in the sense that one adds handles instead of cells. However, in this
situation more efforts are needed to control the normal bundle of the
attaching spheres of the $3$-handles.

\bigskip

\begin{proof}[Proof of Theorem \protect\ref{main2}]
First we may attach $1$- and $2$-handles to the right hand boundary of $%
X\times \lbrack 0,1]$ to obtain an $(n+1)$-dimensional manifold $W_{1}$ such
that $\pi _{1}(W_{1})=G$ and the homomorphism $\pi _{1}(X)\rightarrow \pi
_{1}(W_{1})$ induced by the inclusion $X\hookrightarrow W_{1}$ is $\alpha $.
Note that $W_{1}$ is homotopy equivalent to the complex $W$ constructed in
the proof of Theorem \ref{main}. From the argument of the proof of Theorem %
\ref{main}, we see that $\mathrm{im}\{j\colon H_{2}(W_{1})\rightarrow
H_{2}(W_{1},X)\}$ is a free abelian group, and there is a basis of $\mathrm{%
im}j$ whose elements are spherical, i.e. in the image of 
\begin{equation*}
\pi _{2}(W_{1})\rightarrow H_{2}(W_{1})\rightarrow \mathrm{im}j.
\end{equation*}%
Denote by $X_{1}$ the other boundary component of $W_{1}$. Clearly $\pi
_{2}(W_{1})\cong \pi _{2}(X_{1})$.

If $X$ is spin, then it's well-known that we may choose appropriate framings
of the attaching spheres of the handles such that $W_{1}$ is spin, thus $%
X_{1}$ is also spin. Therefore any embedded $2$-sphere in $X_1$ has trivial
normal bundle and we may attach $3$-handles to obtain $W$ and $Y$ as desired.

In the following, we deal with the case where $X$ is not necessarily spin
but $\ker \{H_{1}(\alpha )\colon H_{1}(\pi _{1}(X);\mathbb{Z})\rightarrow
H_{1}(G;\mathbb{Z})\}$ is torsion free. The key point is to choose
appropriate framings of the attaching spheres of the $2$-handles to ensure
that we may attach the $3$-handles.

Let 
\begin{equation*}
V_{1}=X\times \lbrack 0,1]\cup \bigcup_{i}h_{i}^{1}
\end{equation*}%
be the manifold obtained by attaching $1$-handles and $X^{\prime }$ the
right hand boundary of $V_{1}$. Let 
\begin{equation*}
V_{2}=X^{\prime }\times \lbrack 0,1]\cup \bigcup_{k}h_{k}^{2}
\end{equation*}%
be the result of attaching $2$-handles and $X_{1}$ the the right hand
boundary. Then we get $W_{1}=V_{1}\cup _{X^{\prime }}V_{2}$. In the long
exact sequence 
\begin{equation*}
H_{2}(V_{1},X)\rightarrow H_{2}(W_{1},X)\rightarrow
H_{2}(W_{1},V_{1})\rightarrow H_{1}(V_{1},X)
\end{equation*}%
of the triple $(W_{1},V_{1},X),$ we know that $H_{2}(V_{1},X)=0$ and $%
H_{1}(V_{1},X)$ is torsion free. Therefore, the relative homology group $%
H_{2}(W_{1},X)$ can be viewed as a direct summand of $H_{2}(W_{1},V_{1})%
\cong H_{2}(V_{2},X^{\prime })$. In the long exact sequence 
\begin{equation*}
H_{2}(W_{1})\overset{j}{\rightarrow }H_{2}(W_{1},X)\rightarrow
H_{1}(X)\rightarrow H_{1}(W_{1}),
\end{equation*}%
by assumption, $\mathrm{coker}j\cong \ker \{H_{1}(X)\rightarrow
H_{1}(W_{1})\}=\ker \{H_{1}(\alpha )\colon H_{1}(\pi _{1}(X))\rightarrow
H_{1}(G)\}$ is torsion free. Therefore, the image $\mathrm{im}j$ is
isomorphic to a direct summand of $H_{2}(W_{1},X)$ and hence a direct
summand of $H_{2}(V_{2},X^{\prime })$.

Let the attaching maps of the $2$-handles be 
\begin{equation*}
D_{k}^{2}\times D^{n-1}\supset S^{1}\times D^{n-1}\overset{\varphi _{k}}{%
\hookrightarrow }X^{\prime }.
\end{equation*}%
Then 
\begin{equation*}
X_{1}=(X^{\prime }-\bigcup_{k}\varphi _{k}(S^{1}\times D^{n-1}))\cup
\bigcup_{k}D_{k}^{2}\times S^{n-2}
\end{equation*}%
and we have a canonical basis $\{b_{i}|i=1,\cdots ,m\}$ of $%
H_{2}(V_{2},X^{\prime })$ represented by $D_{k}^{2}\times \{p\}$, where $%
p\in \partial D^{n-1}$ is a fixed point.

Recall that we have elements $x_{1},\cdots ,x_{m}\in \pi _{2}(W_{1})$ such
that $j(x_{1}),\cdots j(x_{m})$ form a basis of $\mathrm{im}j$. Let $%
j(x_{i})=\sum_{k}a_{ik}b_{k}$. We may assume that each $x_i$ is represented
by an embedded $2$-sphere $S^2_i$ in $X_1$, and the intersection of $S^2_i$
with the $2$-handle $h^2_k$ consists of $a_{ik}^{\prime }$ copies of disks $%
D^2_{ik}(1)$, $\cdots$, $D^2_{ik}(a_{ik}^{\prime })$ parallel to the core
disk $D_k^2 \times \{0\}$, where $a_{ik}^{\prime }\equiv a_{ik}\mod 2$.

As seen from the proof of Theorem \ref{main}, we need to attach $3$-handles
along the $2$-spheres $S_{i}^{2}$, which can be done if the normal bundle of
these embedded $2$-spheres are trivial. Note that a stable vector bundle
over $S^{2}$ is determined by its second Stiefel-Whitney class $w_{2}$.
Hence for an embedded $2$-sphere in $W_{1}$, the triviality of its normal
bundle is determined by the evaluation of $w_{2}(W_{1})$ on the homology
class represented by this sphere. That is the following ($\nu $ denotes the
normal bundle of this sphere) 
\begin{equation*}
\langle w_{2}(\nu ),[S_{i}^{2}]\rangle =\langle w_{2}(\nu \oplus
TS^{2}),[S_{i}^{2}]\rangle =\langle w_{2}(W_{1}),x_{i}\rangle .
\end{equation*}

Define a homomorphism $f\colon \mathrm{im}j\rightarrow \mathbb{Z}_{2}$ by $%
f(j(x_{i}))=\langle w_{2}(W_{1}),x_{i}\rangle $. Since $\mathrm{im}j$ is a
direct summand of $H_{2}(V_{2},X^{\prime })$, we can extend $f$ to a
homomorphism $f\colon H_{2}(V_{2},X^{\prime })\rightarrow \mathbb{Z}_{2}$.
Now we rechoose the framing of attaching spheres of the $2$-handles
according to $f(b_{k}),$ as follows. If $f(b_{k})=0$, we keep $\varphi _{k}$
unchanged. If $f(b_{k})=1$, we use the other framing. Denote by $%
W_{1}^{\prime }$ the manifold obtained by using these framing data. Now for
the normal bundle $\nu $, the clutching function along the boundary of $%
D^{2}_{ik}(j)$ ($j=1, \cdots , a_{ik}$) changes if $f(b_k)=1$ and remains
unchanged if $f(b_k)=0$. If a clutching function changes, the evaluation $%
\langle w_2(\nu), [S^2_i]\rangle$ will change by $1$. Therefore 
\begin{equation*}
\langle w_{2}(W_{1}^{\prime }),x_{i}\rangle =\langle
w_{2}(W_{1}),x_{i}\rangle +\sum_{k}a_{ik}^{\prime }f(b_{k})=f(j(x_i)) +
f(j(x_i))=0.
\end{equation*}
Therefore, the normal bundles of the embedded $2$-spheres representing $%
x_{i} $ ($i=1,\cdots ,m$) are trivial and we can attach $3$-handles in the
same manner employed in the proof of Theorem \ref{main}.
\end{proof}

\begin{remark}
The proof only works for the coefficient $R=\mathbb{Z}$. For other
coefficients $R$, even though we know that a basis of $\mathrm{im}j\subset
H_{2}(W_{1},X_{1})\otimes R$ is represented by spheres, we don't know
whether these spheres form a basis of $\mathrm{im}j\subset
H_{2}(W_{1},X_{1}) $ or not. If not, the argument in the above doesn't work
any more.
\end{remark}

Corollary \ref{mcor2} directly follows from Theorem \ref{main2} by noting
that 
\begin{equation*}
H_{1}(f)=H_{1}(\alpha )\colon H_{1}(X)=H_{1}(\pi _{1}(X))\rightarrow
H_{1}(W)=H_{1}(G).
\end{equation*}

\subsection{Applications to homology spheres and Moore manifolds\label{app}}

In this subsection, we give some applications of Theorem \ref{main} and \ref%
{main2}.

Recall the definition of $R$-homology spheres from the introduction.
Corollary \ref{scor1} gives a characterization of the fundamental groups of $%
R$-homology spheres. In order to prove Corollary \ref{scor1}, we need a
lemma. The following result also shows that the CW complex $Y$ in Theorem %
\ref{main} is not unique in general.

\begin{lemma}
\label{killtor}Let $X$ be a simply connected CW complex and $R=\mathbb{Z}/p$
a finite field for some prime $p.$ There exists a simply connected CW
complex $Y$ and an inclusion map $f:X\rightarrow Y$ such that $H_{2}(Y;%
\mathbb{Z})$ is $p$-torsion-free, i.e. $px=0$ implies $x=0$ for $x\in
H_{2}(Y;\mathbb{Z})$, and $f$ induces isomorphism $H_{i}(X;R)\rightarrow
H_{i}(Y;R)$ for any $i\geq 0.$
\end{lemma}

\begin{proof}
Let $S$ be a set of generators for the $p$-torsion elements in $H_{2}(X;%
\mathbb{Z}).$ For each $x\in S,$ attach a $3$-cell to $X.$ We get a new
space $W=X\cup _{\lambda \in S}e_{\lambda }^{3}$. Since $H_{2}(X;R)\cong
H_{2}(X;\mathbb{Z})\otimes_{\mathbb{Z}}R,$ the boundary map $\partial =0$ in
the relative homology exact sequence%
\begin{equation*}
\begin{array}{ccc}
0\rightarrow H_{3}(X;R) & \rightarrow H_{3}(W;R) & \overset{j}{\rightarrow}
H_{3}(W,X;R)\overset{\partial }{\rightarrow }H_{2}(X;R)\rightarrow
H_{2}(W;R)\rightarrow \cdots .%
\end{array}%
\end{equation*}%
Therefore, $j$ is split surjective as an $R$-module homomorphism. Since $%
H_2(W;\mathbb{Z})$ is $p$-torsion-free, the universal coefficient theorem
implies that 
\begin{equation*}
H_{3}(W;R)\cong H_{3}(W;\mathbb{Z})\otimes_{\mathbb{Z}}R.
\end{equation*}%
By the Hurewicz theorem (cf. Hu \cite[Theorem 8.1, p.305]{hu}) and the fact
that tensor product is right exact, the Hurewicz map%
\begin{equation*}
\pi _{3}(W)\otimes_{\mathbb{Z}}R\rightarrow H_{3}(W;\mathbb{Z})\otimes_{%
\mathbb{Z}}R \cong H_{3}(W;R)
\end{equation*}%
is surjective. Using the fact that $\mathbb{Z}/p$ is a $G$-dense ring for
the trivial group (cf. Lemma 2.2 in Ye \cite{Ye}), there exists a set $%
S^{\prime }$ of maps $[g_{\lambda }:S^{3}\rightarrow W]\in \pi _{3}(W)$ such
that the composition%
\begin{equation*}
H_{3}(\vee _{\lambda \in S^{\prime }}S_{\lambda }^{2};R)\rightarrow
H_{3}(W;R)\rightarrow H_{3}(W,X;R)
\end{equation*}%
is isomorphic. For each such a map $g_{\lambda },$ attach a $4$-cell to $W,$
geting a space $Y.$ By the exact relative homology sequence%
\begin{equation*}
\begin{array}{cccc}
\cdots \rightarrow H_{i+1}(Y,X;R) & \rightarrow & H_{i}(W,X;R) & \rightarrow
H_{i}(Y,W;R)\rightarrow H_{i}(Y,X;R)\rightarrow \cdots%
\end{array}%
\end{equation*}%
and a similar diagram chase as that in the proof Theorem \ref{main}, we see
that $H_{i}(Y,X;R)=0$ for any $i\geq 0.$ This shows that the inclusion $%
X\hookrightarrow Y$ induces a homology equivalence with coefficients $R$.
\end{proof}

\bigskip

\begin{proof}[Proof of Corollary \protect\ref{scor1}]
The first part follows directly from Theorem \ref{main} with $f:\pi
_{1}(X)\rightarrow G$ the trivial group homomorphism. When $n>2,$ the
fundamental group $G$ of an $R$-homology $n$-sphere satisfies the condition
that $H_{1}(G;R)=H_{2}(G;R)=0,$ by the Hopf exact sequence (cf. \cite[Lemma
2.2]{hig}).

The $R$-homology $1$-sphere is a special kind of a generalized Moore space $%
M(G,1;R)$ defined in Ye \cite{Ye}. It is proved that a group $G$ is the
fundamental group of a Moore space $M(G,1;R)$ if and only if $H_{2}(G;R)=0$
(cf. \cite[Proposition 4.6.]{Ye}). It follows that a group $G$ is the
fundamental group of an $R$-homology $1$-sphere if and only if $H_{1}(G;R)=R$
and $H_{2}(G;R)=0$.

We consider the case of $n=2.$ By the Hopf exact sequence again, we see that
the condition that $H_{1}(G;R)=0$ and $H_{2}(G;R)$ is a quotient of $R$ is
necessary. Conversely, let $X=K(G,1)$ be a classifying space of $G$ and $%
\alpha :G\rightarrow 1$ a trivial group homomorphism. By Theorem \ref{main},
there is a simply connected CW complex $Y$ and a map $f:X\rightarrow Y$
inducing an $R$-homology equivalence. Note that the coefficients $R$ is a
principal ideal domain and there is an $R$-epimorphism 
\begin{equation*}
R\rightarrow H_{2}(G;R)\cong H_{2}(Y;R)\cong \pi _{2}(Y)\otimes _{\mathbb{Z}%
}R.
\end{equation*}%
By applying Lemma \ref{killtor} if necessary, we may assume that $\pi
_{2}(Y)\cong H_{2}(Y;\mathbb{Z})$ is a cyclic group. Choose $\eta
:S^{2}\rightarrow Y$ as a generator of $\pi _{2}(Y)$. Viewing $f$ as a
fibration by replacing $X$ by the path space $E_{f}$ (still denoted by $X$
without confusions), we let $K$ denote the pullback the following diagram%
\begin{equation*}
\begin{array}{ccc}
K & \rightarrow  & X \\ 
\downarrow  &  & \downarrow f \\ 
S^{2} & \overset{\eta }{\rightarrow } & Y.%
\end{array}%
\end{equation*}%
Denote by $F$ the homotopy fiber of $f.$ By the commutative diagram%
\begin{equation*}
\begin{array}{ccccccc}
\cdots \rightarrow \pi _{2}(S^{2}) & \rightarrow  & \pi _{1}(F) & 
\rightarrow  & \pi _{1}(K) & \rightarrow  & \pi _{1}(S^{2})=1 \\ 
\downarrow  &  & \downarrow  &  & \downarrow  &  & \downarrow  \\ 
\cdots \rightarrow \pi _{2}(Y) & \rightarrow  & \pi _{1}(F) & \rightarrow  & 
\pi _{1}(X)=G & \rightarrow  & \pi _{1}(Y)=1,%
\end{array}%
\end{equation*}%
we see that $\pi _{1}(K)\rightarrow G$ is an isomorphism. By the Serre
spectral sequence, we see that $H_{\ast }(F;R)=H_{\ast }(\mathrm{pt};R).$
Using the Serre spectral sequence again, the map $K\rightarrow S^{2}$
induces an $R$-homology equivalence. This finishes the proof.
\end{proof}

\begin{remark}

\begin{enumerate}
\item[(i)] The existence of $\mathbb{Z}$-homology 2-sphere is actually
already contained in Dror \cite[proof of Theorem 3.2, p.122]{dror}.

\item[(ii)] Although Kervaire \cite{kar} proves that every finitely
presented group $G$ with $H_{1}(G;\mathbb{Z})=H_{2}(G;\mathbb{Z})=0$ could
be realized as the fundamental group of a $\mathbb{Z}$-homology sphere $%
M^{n} $ (a closed manifold) when $n\geq 5,$ Hausmann and Weinberger \cite{hw}
show that it is not true for $n=4.$

\item[(iii)] The authors don't know whether every finitely presented group
with the condition in Corollary \ref{scor1} (ii) could be realized as a
finite $R$-homology $2$-sphere.
\end{enumerate}
\end{remark}

Recall from \cite{ann} that for a given integer $n\geq 1$ and a group $G$
(abelian if $n\geq 2$), a Moore space $M(G;n)$ is a CW complex $X$ such that
the homotopy group $\pi _{j}(X)=0$ for $j<n$, $\pi _{n}(X)=G$ and the
homology group $H_{i}(X;Z)=0$ for each $i>n.$ As analogues of Moore spaces,
we define Moore manifolds as follows. Let $k$ be a positive integer and $G$
a finitely presented group. When $k\geq 2,$ we assume further that $G$ is
abelian.

\begin{definition}
Let $n,k$ be two positive integers. An $n$-dimensional Moore manifold $%
M^{n}(G,k)$ is an orientable closed manifold $X$ such that for any integer $%
i<k,$ the homotopy group $\pi _{i}(X)=0,$ $\pi _{k}(X)=G$ and for any
integer $k<i\leq (n+1)/2,$ the homology group $H_{i}(X;\mathbb{Z})=0.$
\end{definition}

When $k>[(n+1)/2],$ by Poincar\'{e} duality, Moore manifold $M^{n}(G,k)$
only exists when $G=1,$ the trivial group. In this case, $M^{n}(G,k)$ is the
standard sphere. Therefore, in the remainder of this article, we always
assume $k\leq (n+1)/2.$

Similar to the existence of Moore spaces $M(G,1)$ in Varadarajan \cite{ann},
we give a characterization of Moore manifolds $M^{n}(G,1),$ as follows.

\begin{proposition}
\label{prop}Let $n\geq 5$ and $G$ a finitely presented group. There exists a
Moore manifold $M^{n}(G,1)$ if and only if $H_{2}(G;\mathbb{Z})=0.$
\end{proposition}

\begin{proof}
The necessary condition follows easily Hopf's exact sequence%
\begin{equation*}
\pi _{2}(M^{n}(G,1))\rightarrow H_{2}(M^{n}(G,1);\mathbb{Z})\rightarrow
H_{2}(G;\mathbb{Z})\rightarrow 0.
\end{equation*}%
Conversely when $H_{2}(G;\mathbb{Z})=0,$ we choose $X=S^{n}$ and $\alpha
:1\rightarrow G$ the trivial group homomorphism. By Theorem \ref{main2},
there exists an orientable closed manifold $Y,$ which is obtained from $X$
by adding 1-handles, 2-handles and 3 handles, such that $\pi _{1}(Y)=G$ and
the inclusion $X\hookrightarrow W,$ the cobordism between $X$ and $Y,$
induces that for any integer $q\geq 2,$ the relative homology group $%
H_{q}(W,X;\mathbb{Z})=0.$ By the universal coefficients theorem and Poincar%
\'{e} duality, for each $q\geq 2,$ there is an isomorphism $H^{q}(W,X;%
\mathbb{Z})=H_{n+1-q}(W,Y;\mathbb{Z})=0.$ This implies that for any integer $%
2\leq i\leq n-2,$ the homology group $H_{i}(Y;\mathbb{Z})=H_{i}(X;\mathbb{Z}%
)=0.$ By the assumption that $n\geq 5,$ we have $n-2\geq \lbrack (n+1)/2].$
Therefore, such $Y$ is a Moore manifold $M^{n}(G,1).$
\end{proof}

\begin{proposition}
Let $n\geq 5,$ $k<(n-1)/2$ and $G$ a finitely generated abelian group. There
exists a Moore manifold $M^{n}(G,k).$
\end{proposition}

\begin{proof}
Without loss of generality, suppose that $G=\mathbb{Z}/t$ for some integer $%
t.$ When $t=0,$ assume $G=\mathbb{Z}$. The general Moore manifold can be
obtained as connected sum of such manifolds $M^{n}(\mathbb{Z}/t,k).$ Take $%
X=S^{k}\times S^{n-k}.$ Let $f:S^{k}\rightarrow X$ be an embedding
representing the element $[t]\in \mathbb{Z=\pi }_{k}(X).$ Since $X$ is
parallelizable, $f(S^{k})$ has a trivial normal bundle in $X.$ Extend $f$ to
an embedding $\tilde{f}:S^{k}\times D^{n-k}\rightarrow X.$ Doing surgery on $%
X$ along $\tilde{f}$, the resulting manifold is denoted by $Y$. Suppose that
the surgery trace is $W.$ It is not hard to see that $H_{k}(W;\mathbb{Z})=G$%
, and the homology group $H_{i}(W)=0$ for any integer $k<i<(n+1)/2$. Since $%
W\simeq Y\cup e^{n-k}$ and $n-k>(n+1)/2>k+1,$ the manifold $Y$ has the same
homology groups as $W$ up to the middle dimension. This shows that $Y$ is a
Moore space $M^{n}(G,k).$
\end{proof}

\begin{remark}

\begin{enumerate}
\item For an integer $k$ close to $(n+1)/2$ the manifold $M^{n}(G,k)$ may
not exits, see the Corollary before Lemma F in Barden \cite{bar}.

\item Hausmann and Weinberger \cite{hw} constructed a superperfect group $G$
for which any $4$-manifold $Y$ with $\pi _{1}(Y)=G$ satisfies $\chi (Y)>2.$
This implies that Proposition \ref{prop} does not hold for $n=4.$
\end{enumerate}
\end{remark}

As another application of Theorem \ref{main2}, the following result is an
improvement of Corollary 1.3 in \cite[Theorem 1.1]{Ye2}.

\begin{corollary}
\label{integral}Let $n\geq 5$ and $X$ be a closed $n$-dimensional manifold
with fundamental group $\pi $ and $N$ a normal subgroup of $\pi .$ The
following are equivalent:

\begin{enumerate}
\item[(i)] There is a closed manifold $Y$ obtained from $X$ by adding $2$%
-handles and $3$-handles with $\pi _{1}(Y)=\pi /N$ such that for any $q\geq
0 $ there is an isomorphism%
\begin{equation*}
H_{q}(Y;\mathbb{Z})\cong H_{q}(X;\mathbb{Z}).
\end{equation*}

\item[(ii)] The group $N$ is normally generated by some finite number of
elements and is a relatively perfect subgroup of $\pi ,$ \textsl{i.e.} $[\pi
,N]=N.$
\end{enumerate}
\end{corollary}

\begin{proof}
Compared with Corollary 1.3 of \cite{Ye2}, we drop the condition that $X$ is
spin here, since $H_{1}(\pi ;\mathbb{Z})\rightarrow H_{1}(\pi /N;\mathbb{Z})$
is an isomorphism and Theorem \ref{main2} applies.
\end{proof}

\section{Quillen's plus construction with a given Whitehead torsion\label{qm}%
}

In this section, we introduce Quillen's plus constructions with given
Whitehead torsions for both CW complexes and manifolds. Theorem \ref{th3} is
proved at the end of this section.

\subsection{Plus construction with torsions for CW complexes}

Let $X$ be a CW-complex and $P\lhd \pi _{1}(X)$ a perfect normal subgroup,
normally generated by finitely many elements. Then Quillen's plus
construction is a CW-complex $X^{+}$ containing $X$ as a subcomplex such
that $i_{\ast }\colon \pi _{1}(X)\rightarrow \pi _{1}(X^{+})$ is the
projection $\pi _{1}(X)\rightarrow \pi _{1}(X)/P$ and the pair $(X^{+},X)$
is homologically acyclic, i.e. $H_{\ast }(X^{+},X;A)=0$ for any $\pi
_{1}(X)/P$-module $A$. Especially, we have $H_{\ast }(\widetilde{X^{+}},%
\overline{X})=0$, where $\widetilde{X^{+}}$ is the universal cover of $X^{+}$
and $\overline{X}$ is the corresponding covering space of $X$. Therefore,
there is a well-defined torsion $\tau (X^{+},X)\in \mathrm{Wh}(G)$ ($G=\pi
_{1}(X)/P$) of the pair $(X^{+},X)$ (cf. Remark 2 of \cite[p.~378]{Milnor}).

\begin{theorem}
\label{torsion} Given an element $\tau \in \mathrm{Wh}(G)$, there exists a
plus construction $X^{+}$ of $X$ such that $\tau (X^{+},X)=\tau $. If there
is another $X_{1}^{+}$ with the same property, then there is a simply
homotopy equivalence $f\colon X^{+}\rightarrow X_{1}^{+}$ which is homotopic
to the identity on $X$.
\end{theorem}

This is a stronger version of the existence and uniqueness of the plus
construction (cf. \cite[Theorem 5.2.2]{Rosenberg}).

\begin{proof}
We just need to modify the ordinary plus construction to take into account
the torsion issue. Let $\tau \in \mathrm{Wh}(G)$ be represented by a matrix $%
A=(a_{ij})$ of size $N$ for a larger integer $N$. Let $Y$ be obtained by
attaching $k$ $2$-cells $e_{i}^{2}$ on $X$ to have the fundamental group $G$
and $(N-k)$ $2$-cells with trivial attaching maps. Let $\widetilde{Y}$ be
the universal cover of $Y$ and $\overline{X}$ be the corresponding cover of $%
X$. Then the relative homology groups of $(\widetilde{Y},\overline{X})$
concentrate in dimension $2$ and the homomorphism 
\begin{equation*}
j\colon \pi _{2}(Y)\cong H_{2}(\widetilde{Y})\rightarrow H_{2}(\widetilde{Y},%
\overline{X})=\bigoplus_{i=1}^{N}\mathbb{Z}[G][e_{i}^{2}]
\end{equation*}%
is surjective since $H_{1}(\overline{X})=H_{1}(\pi )=0$. Therefore we may
choose $x_{1},\ldots ,x_{N}\in \pi _{2}(Y)$ such that when expressed in the
canonical basis $[e_{i}^{2}]$, the coefficients of $j(x_{i})$ are the $i$-th
row of $A$. Using these $x_{i}$ as attaching data we form 
\begin{equation*}
X^{+}=Y\cup \cup e_{i}^{3}.
\end{equation*}%
Now the chain complex $C_{\ast }(\widetilde{X^{+}},\overline{X})$
concentrates in dimension $2$ and $3$, and 
\begin{equation*}
C_{3}(\widetilde{X^{+}},\overline{X})=H_{3}(\widetilde{X^{+}},\widetilde{Y}%
)=\bigoplus \mathbb{Z}[G][e_{i}^{3}]
\end{equation*}%
\begin{equation*}
C_{2}(\widetilde{X^{+}},\overline{X})=H_{2}(\widetilde{Y},\overline{X}%
)=\bigoplus \mathbb{Z}[G][e_{i}^{2}]
\end{equation*}%
and the boundary map $\partial \colon C_{3}(\widetilde{X^{+}},\overline{X}%
)\rightarrow C_{2}(\widetilde{X^{+}},\overline{X})$ is just the boundary map 
$\partial \colon H_{3}(\widetilde{X^{+}},\widetilde{Y})\rightarrow H_{2}(%
\widetilde{Y},\overline{X})$ in the long exact sequence of the triple $(%
\widetilde{X^{+}},\widetilde{Y},\overline{X})$. Hence by construction, the
pair $(X^{+},X)$ is homologically acyclic and the torsion of $C_{\ast }(%
\widetilde{X^{+}},\overline{X})$ is represented by $A$, which equals to $%
\tau $.

For the uniqueness of $X^{+}$, it is shown that there exists a homotopy
equivalence $f\colon X^{+}\rightarrow X_{1}^{+}$ which is homotopic to the
identity on $X$ (cf. \cite[Theorem 5.2.2]{Rosenberg}). There is a short
exact sequence of chain complexes 
\begin{equation*}
0\rightarrow C_{\ast }(\widetilde{X^{+}},\overline{X})\rightarrow C_{\ast }(%
\widetilde{X_{1}^{+}},\overline{X})\rightarrow C_{\ast }(\widetilde{X_{1}^{+}%
},\widetilde{X^{+}})\rightarrow 0
\end{equation*}%
obtained from the triple $(X_{1}^{+},X^{+},X)$. From the additivity of the
Whitehead torsion \cite[Theorem 3.1]{Milnor}, we have $\tau
(X_{1}^{+},X)=\tau (X^{+},X)+\tau (X_{1}^{+},X^{+})$. Therefore $\tau
(X_{1}^{+},X^{+})=0$, which implies that $f$ is a simple homotopy
equivalence.
\end{proof}

\subsection{Embedding manifold plus construction with torsions}

In the case that $X$ is a manifold $M^{n}$ ($n\geq 5$), it is shown in \cite[%
Theorem 4.1]{gt} that the plus construction can be done in the world of
manifolds and one obtains a one-sided $h$-cobordism $(W;M,M_{2})$ (i.e. the
inclusion $M_{2}\hookrightarrow W$ is a homotopy equivalence) such that the
Whitehead torsion of $(W,M_{2})$ is trivial. In this section we generalize
the manifold plus construction as following.

\begin{theorem}
\label{mfd} Let $M^{n}$ be a manifold of dimension $n\geq 5$, $P\lhd \pi
_{1}(M)$ a perfect normal subgroup normally generated by finitely many
elements. Let $\tau \in \mathrm{Wh}(G)$ ($G=\pi _{1}(M)/P$) be an arbitrary
element. Then there is a one-sided $h$-cobordism $(W;M,M_{2})$ such that $W$
is the plus construction of $M$ corresponding to $P$ and the Whitehead
torsion $\tau (W,M_{2})=\tau $. Furthermore, $W$ is unique up to
diffeomorphism rel $M $.
\end{theorem}

Actually the existence part of this theorem can be deduced from a
combination of \cite[Theorem 4.1]{gt} and \cite[Existence Theorem 11.1]%
{Milnor}. However, the proof given below shows that in the process of
attaching handles, we can control the homotopy type and the Whitehead
torsion simultaneously. Therefore, it can be viewed as a generalization both
of \cite[Theorem 4.1]{gt} and of \cite[Existence Theorem 11.1]{Milnor}.

First we need to generalize the technique of \cite{gt} so that it is
applicable to the construction given in the proof of Theorem \ref{torsion}.

Let $M^{n}$ be a smooth manifold of dimension $n\geq 5$, $S \subset M$ be an
embedded $S^{1}$ which is null-homologous. Since $\langle
w_{1}(M),[S]\rangle =0$, the normal bundle of $S$ in $M$ is trivial. There
are essentially two framings of this normal bundle. Let $W^{n+1}$ be the
result of attaching a $2$-handle to the right hand boundary of $M\times
\lbrack 0,1]$ using $S$ as the attaching sphere, then $W \simeq M \cup e^2$,
and the natural projection $j \colon H_2(W) \to H_2(W,M) \cong \mathbb{Z}$
is surjective. Let $B\subset H_{2}(W)$ be any direct summand of $H_{2}(W)$,
mapped isomorphically to $H_{2}(W,M)$ under $j$.

\begin{lemma}
\label{spin} There exists a framing of the normal bundle of $S$ such that
for the resulting manifold $W$, the evaluation of the second Stiefel-Whitney
class $w_{2}(W)$ on $B$ is trivial. (Since $W$ is canonically homotopy
equivalent to $M\cup e^{2}$, we may identify the homology groups of $W$
obtained using different framings.)
\end{lemma}

\begin{proof}
We start from choosing one framing and get an embedding $\varphi \colon
S^{1}\times D^{n-1}\hookrightarrow M$ with $S=S^{1}\times \{0\}$ and attach
a $2$-handle $h^{2}$ via $\varphi $ 
\begin{equation*}
W=M\times \lbrack 0,1]\cup _{\varphi }D^{2}\times D^{n-1}.
\end{equation*}%
Let $M_{0}=M-\varphi (S^{1}\times D^{n-1})$ and $M_{1}=M_{0}\cup _{\varphi
}D^{2}\times S^{n-2},$ the other end of $W$. Clearly, we have that $%
H_{2}(W)=H_{2}(M_{1})$.

The Hurewicz map $\Omega _{2}^{SO}(W)\rightarrow H_{2}(W)$ from the
cobordism group of $W$ to the homology group of $W$ is surjective (easily
seen by a standard Atiyah-Hirzebruch spectral sequence argument), therefore
a generator of $B$ is represented by an embedded closed oriented surface $%
F^{2}\hookrightarrow M_{1}\subset W$. After an isotopy of $F^2$, we may
assume that the intersection of $F^2$ with $\varphi (D^2 \times S^{n-2})$
consists of $m$ disks $D^2 \times \{x_{1}\}, \cdots, D^2 \times \{x_{m}\}$ ($%
x_{i} \in S^{n-2}$) parallel to the core disk $D^2 \times \{0\}$. By
surgering away extraneous pairs of algebraically opposite $2$-disks, we get
a new surface, still denoted by $F$, whose intersection with $\varphi (D^2
\times S^{n-2})$ is $D^2 \times \{x\}$.

Let $\nu$ be the normal bundle of the embedded surface $F$. It's known that
an orientable stable vector bundle $\nu$ over a closed surface $F$ is
trivial if and only if $\langle w_2(\nu), [F] \rangle =0$. Also since the
stable tangent bundle of a closed orientable surface is trivial, we have
that 
\begin{equation*}
\langle w_{2}(W),[F]\rangle =\langle w_{2}(TF \oplus \nu ),[F]\rangle
=\langle w_{2}(\nu ),[F]\rangle .
\end{equation*}%
Therefore if $\nu$ is trivial, then we are done. If $\nu$ is nontrivial,
then we use the other framing of $S$. In this case, along the boundary of $%
D^2 \times \{x\}$, the clutching function of the normal bundle $\nu $
changes, the new normal bundle is trivial.
\end{proof}

\medskip

\begin{proof}[Proof of Theorem \protect\ref{mfd}]
Let $\bar{\tau}$ be the conjugate of $\tau $ (for the definition, see
Section 6 of \cite{Milnor}). First we attach $2$-handles to $M\times \lbrack
0,1]$ to kill $P$. We also attach some trivial $2$-handles such that the
total number of $2$-handles is $N$ if $(-1)^{n}\bar{\tau}\in \mathrm{Wh}(G)$
is represented by a matrix $A$ of size $N$. Denote by $W_{1}$ the surgery
trace and by $M_{1}$ the right hand boundary of $W_{1}$. Suppose that $%
\widetilde{W}_{1}$ is the universal covering space of $W_{1}$ and $\overline{%
M}$ is the corresponding covering space of $M.$ As in the proof of Theorem %
\ref{torsion}, we have a surjection 
\begin{equation*}
\widetilde{j}\colon \pi _{2}(W_{1})\cong H_{2}(\widetilde{W}_{1})\rightarrow
H_{2}(\widetilde{W}_{1},\overline{M})=\bigoplus_{i=1}^{N}\mathbb{Z}%
[G][h_{i}^{2}].
\end{equation*}%
We choose $x_{1},\cdots ,x_{N}\in \pi _{2}(W_{1})$ such that the
coefficients of $\widetilde{j}(x_{i})$ in the basis $[h_{i}^{2}]$ consist of
the $i$-th row of $A$. $\pi _{2}(M_{1})\cong \pi _{2}(W_{1})$.

Note that a stable vector bundle $\xi $ over $S^{2}$ is determined by its
second Stiefel-Whitney class $w_{2}(\xi )$. Hence for an embedded $2$-sphere
in $W_{1}$, the triviality of its normal bundle is determined by the
evaluation of $w_{2}(W)$ on the homology class represented by this sphere.
From the commutative diagram%
\begin{equation*}
\begin{array}{ccc}
H_{2}(\widetilde{W}_{1}) & \overset{\bar{j}}{\rightarrow } & H_{2}(%
\widetilde{W}_{1},\bar{M}) \\ 
\downarrow &  & \downarrow \\ 
H_{2}(W_{1}) & \overset{j}{\rightarrow } & H_{2}(W_{1},M),%
\end{array}%
\end{equation*}
it's seen that under the Hurewicz map, the image of $x_{i}$ ($i=1,\cdots ,N$%
) generate a direct summand $B$ of $H_{2}(W_{1})$, which is mapped
isomorphically to $H_{2}(W_{1},M)$ under $j$. Now by Lemma \ref{spin}, we
may choose appropriate framings of the attaching spheres of the $2$-handles
such that the evaluation of $w_{2}(W_{1})$ on $B$ is zero.

Therefore, we may attach $3$-handles to $M_{1}$ (since $\pi _{2}(M_{1})\cong
\pi _{2}(W_{1})$) using embedded $2$-spheres representing $x_{i}$ ($%
i=1,\cdots ,N$). Denote by $W$ the resulting manifold with right hand
boundary $M_{2}$. From the construction, we see that $W\simeq M^{+}$. Hence, 
$(W,M)$ is homologically acyclic. By Poincar\'{e} duality, we get $H_{\ast }(%
\widetilde{W},\widetilde{M_{2}})=0$, which implies $W\simeq M_{2}$. Also
from the construction it's seen that $\tau (W,M)=[A]=(-1)^{n}\bar{\tau}$. By
the duality of Whitehead torsion (cf. \cite[p.~394]{Milnor}), we have $\tau
(W,M_{2})=\tau $.

In order to prove the uniqueness of $W,$ we modify the construction in \cite[%
p.~197]{fq}. Let $(W^{\prime };M,M^{\prime })$ be another such one-sided $h$%
-cobordism with base $M$. Suppose that $X=W\cup _{M}W^{\prime }$. Then $X$
is an $h$-cobordism with two ends $M_{2}$ and $M^{\prime }$. For the
Whitehead torsions, we have that 
\begin{equation*}
\begin{array}{rcl}
\tau (M_{2}\rightarrow X) & = & \tau (M_{2}\rightarrow W)+\tau (W\rightarrow
X) \\ 
& = & \tau +\tau (M\rightarrow W^{\prime }) \\ 
& = & \tau +(-1)^{n}\bar{\tau}.%
\end{array}%
\end{equation*}

Suppose that $(V;X,X_{1})$ is an $h$-cobordism rel boundary with base $X$
and Whitehead torsion $\tau (X\rightarrow V)=(-1)^{n-1}\bar{\tau}$ (see the
figure below). 

\begin{figure}[ht]
\centering
\includegraphics[height=5.5cm]{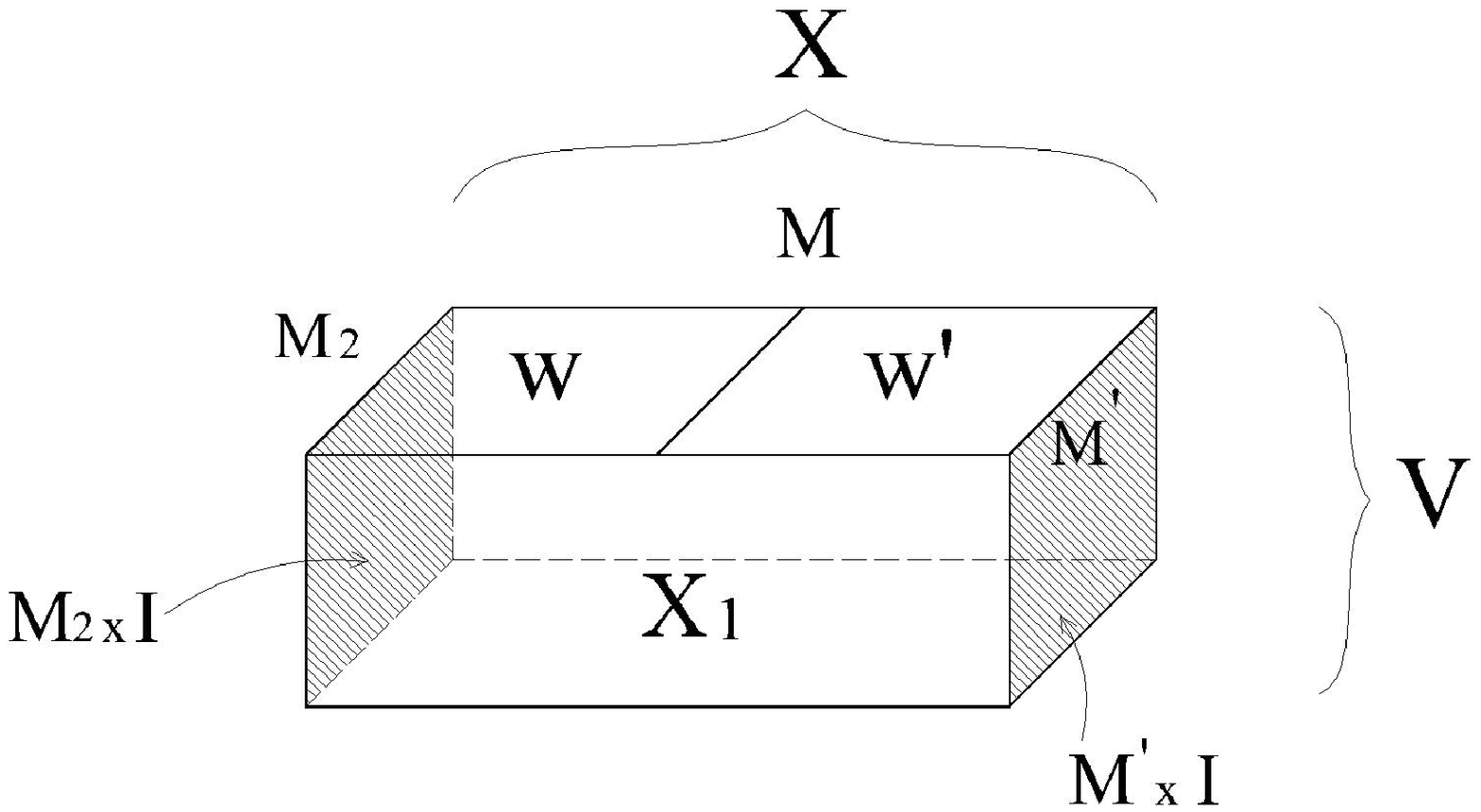} 
\end{figure}
Then we have that 
\begin{equation*}
\tau (M_{2}\rightarrow V)=\tau (M_{2}\rightarrow X)+\tau (X\rightarrow
V)=\tau
\end{equation*}%
and 
\begin{equation*}
\tau (M_{2}\rightarrow V)=\tau (M_{2}\rightarrow X_{1})+\tau
(X_{1}\rightarrow V)=\tau (M_{2}\rightarrow X_{1})+\tau .
\end{equation*}%
These imply that $\tau (M_{2}\rightarrow X_{1})=0$ and $X_{1}$ is an $s$%
-cobordism. On the other hand, since the Whitehead torsion 
\begin{equation*}
\tau (W\rightarrow V)=\tau (W\rightarrow X)+\tau (X\rightarrow V)=0,
\end{equation*}%
$V$ is an $s$-cobordism relative to the boundary from $W$ to $W^{\prime }$.
Therefore, by the $s$-cobordism theorem, the diffeomorphism $X_{1}\cong
M_{2}\times \lbrack 0,1]$ extends to a diffeomorphism $V\cong W\times
\lbrack 0,1]$ and $W$ is diffeomorphic to $W^{\prime }$ relative to $M$.
\end{proof}

The embedding plus construction for manifolds is considered by Guilbault and
Tinsley \cite{gt,gt2}. This is important to their study of ends of
non-compact manifolds. We give an embedding plus construction with a given
Whitehead torsion as follows.

\begin{theorem}
Let $W^{n}$ be a connected manifold of dimension $n\geq 6$ and $M$ a closed
component of the boundary of $W.$ Suppose that $P$ is a normal subgroup of
the kernel $\mathrm{\ker }(\pi _{1}(M)\rightarrow \pi _{1}(W)),$ which is
normally generated by a finite set of elements. Then for any element $\tau
\in \mathrm{Wh}(\pi _{1}(M)/P),$ there exists a one-sided $h$-cobordism $%
(W^{\prime };M,M^{\prime })$ embedded in $W$ fixing $M$ such that $\pi
_{1}(W^{\prime })=\pi _{1}(M)/P$ and $\tau (W^{\prime },M^{\prime })=\tau .$
\end{theorem}

\begin{proof}
By Theorem \ref{mfd}, there exists a one-sided $h$-cobordism $(W^{\prime
};M,M^{\prime })$ such that $\tau (W^{\prime },M^{\prime })=\tau .$
According to Theorem 11.1 of Milnor \cite{Milnor}, there exists a cobordism $%
(W_{1};M^{\prime },N)$ such that $\tau (W_{1},N)=-\tau .$ Glue $W^{\prime }$
and $W_{1}$ together along $M^{\prime }$ to get a new manifold $W^{\prime
}\cup _{M^{\prime }}W_{1}.$ Then $\tau (W^{\prime }\cup _{M^{\prime
}}W_{1},N)=0$. Note that $(W^{\prime }\cup _{M^{\prime }}W_{1};M,N)$ is a
one-sided $h$-cobordism with the inclusion map $N\hookrightarrow W^{\prime
}\cup _{M^{\prime }}W_{1}$ a simple homotopy equivalence. By Theorem 1.1 in 
\cite{gt}, there is an embedding $W^{\prime }\cup _{M^{\prime
}}W_{1}\rightarrow W$ fixing $M.$ As $W^{\prime }$ is a subset of $W^{\prime
}\cup _{M^{\prime }}W_{1},$ we finish the proof.
\end{proof}

\bigskip

\begin{proof}[Proof of Theorem \protect\ref{th3}]
For each one-sided $h$-cobordism $(W;M,N),$ the inclusion map $%
M\hookrightarrow W$ induces a homology equivalence with coefficients $%
\mathbb{Z[}\pi _{1}(W)]$ by Poincar\'{e} duality. This shows that the
inclusion $M\hookrightarrow W$ is a Quillen plus construction. Therefore,
the kernel $\mathrm{\ker (}\pi _{1}(M)\rightarrow \pi _{1}(W))$ is a perfect
normal subgroup. Since both $\pi _{1}(M)$ and $\pi _{1}(N)$ are finitely
presented, this kernel is normally finitely generated. Assign $(W;M,N)$ the
Whitehead torsion $\tau (W,N)\in \mathrm{Wh}(\pi _{1}(W)).$ Since a
diffeomorphism has trivial Whitehead torsion, this gives a well-defined map
from $S_{h}(M)$ to the right hand. Theorem \ref{mfd} shows that this map is
both surjective and injective.
\end{proof}

Hua Loo-Keng Key Laboratory of Mathematics, Chinese Academy of Sciences,
Beijing, 100190, China

E-mail: suyang@math.ac.cn

\bigskip

Department of Mathematical Sciences, Xi'an Jiaotong-Liverpool University,
Suzhou, Jiangsu 215123, China

E-mail: Shengkui.Ye@xjtlu.edu.cn

\end{document}